\documentclass[twoside]{amsart}
\usepackage{amsmath,amsfonts,amsthm,mathrsfs,amscd}
\usepackage[unicode]{hyperref}
\usepackage[numeric,initials]{amsrefs}
\usepackage{dcpic,pictexwd}
\usepackage{enumitem}

\usepackage{amsmath,amsfonts,amsthm,mathrsfs,amscd,stmaryrd}
\usepackage[numeric,initials]{amsrefs}

\newcommand{\norm}[1]{\lVert#1\rVert}

\newcommand{\abs}[1]{\lvert #1 \rvert}

\newcommand{\Hil}{\mathcal{H}}

\newcommand{\NN}{\mathbb{N}}
\newcommand{\ZZ}{\mathbb{Z}}

\newtheorem{thm}{Theorem}
\newtheorem{lem}[thm]{Lemma}

\newtheorem*{lem*}{Lemma}

\newtheorem*{thm*}{Theorem}

\theoremstyle{definition}

\theoremstyle{remark}

\theoremstyle{plain}

\DeclareMathOperator{\im}{im}

\DeclareMathOperator{\id}{id}

\renewcommand{\Re}{\operatorname{Re}}

\newcommand{\cl}[1]{\overline{#1}}

\newcommand{\define}[1]{\emph{#1}}

\def\overbar#1#2#3{{%
	\setbox0=\hbox{$#1$}%
	\dimen0=\wd0
	\advance\dimen0 by -#2 
	\vbox {\nointerlineskip \moveright #3 \vbox{\hrule height 0.3pt width \dimen0}%
		\nointerlineskip \vskip 1.5pt \box0}%
}}

\newcommand{\N}{\mathbb{N}}
\newcommand{\Z}{\mathbb{Z}}
\newcommand{\Q}{\mathbb{Q}}
\newcommand{\Zn}[1][n]{\Z^{#1}}

\newcommand{\set}[2]{\bigl\{#1 \thinspace\big\vert\thinspace #2 \bigr\}}
\newcommand{\Set}[2]{\Bigl\{#1 \thinspace\Big\vert\thinspace #2 \Bigr\}}
\newcommand{\seq}[2]{( #1 )_{#2}}

\newtheorem{thmsec}{Theorem}[section]
\numberwithin{equation}{section}

\renewcommand{\abs}[1]{\left| #1 \right|}

\renewcommand{\norm}[1]{\bigl\| #1 \bigr\|}
\newcommand{\Norm}[1]{\Bigl\| #1 \Bigr\|}
\newcommand{\inner}[2]{\bigl\langle #1, #2 \bigr\rangle}
\newcommand{\Inner}[2]{\Bigl\langle #1, #2 \Bigr\rangle}

\newcommand{\msp}{(X, \mathcal{B}, \mu)}
\newcommand{\ms}[1]{\mu(#1)}
\newcommand{\Ms}[1]{\mu\biggl(#1\biggr)}
\newcommand{\Lzwo}{\mathit{L}^2(X, \mu)}

\newcommand{\plim}[1]{\mathop{#1\text{-}}\negthinspace\mathop{\mathrm{lim}}}
\newcommand{\IP}[1]{\operatorname{IP}\bigl( #1 \bigr)}
\newcommand{\FU}[2]{\operatorname{IP}( #1 )_{#2}}
\newcommand{\clA}{\overbar{A}{3pt}{3pt}}
\newcommand{\phis}[1][\phi]{#1_{\ast}}

\newcommand{\polgrp}[2]{#1 \lbrack #2 \rbrack}
\newcommand{\polint}[1]{\polgrp{\mathrm{Int}}{#1}}
\newcommand{\z}[1][\thinspace]{\vec{z}_{#1}}
\newcommand{\w}[1][]{\vec{w}_{#1}}
\renewcommand{\a}[1][]{\vec{a}_{#1}}
\renewcommand{\b}[1][]{\vec{b}_{#1}}
\renewcommand{\c}[1][]{\vec{c}_{#1}}
\renewcommand{\alph}[1][]{\vec{\alpha}_{#1}}
\newcommand{\dsum}[2]{\Delta \negthinspace ^{#1} #2}

\newcommand{\unop}[1]{\prod_{i=1}^{m} U_i^{#1}}

\newcommand{\mspt}[1]{\Bigl(\prod_{i=1}^{m} T_i^{#1}\Bigr)^{-1}}

\newcommand{\betaFun}{\beta\mathcal{F}^{\mathit{un}}}

\newcommand{\nbull}{n_{\bullet}}
\newcommand{\mbull}{m_{\bullet}}

\begin{document}

\title{Idempotent ultrafilters and polynomial recurrence}
\author[C.~Schnell]{Christian Schnell}
\address{The Ohio State University\\
231 West 18th Avenue\\
Columbus, OH 43210}
\email{schnell@math.ohio-state.edu}


\begin{abstract}
We give a new proof of a polynomial recurrence result due to Bergelson, Furstenberg,
and McCutcheon, using idempotent ultrafilters instead of IP-limits.
\end{abstract}
\maketitle

\section*{Introduction} \label{sec:0}

In the thirty or so years since H.~Furstenberg reproved Szemer\'edi's
theorem using methods from ergodic theory, many striking discoveries have been
made in the area now known as \emph{Ergodic Ramsey theory}. Perhaps the most
surprising of these is the discovery that recurrence results can be obtained
for polynomial sets, meaning sets of values of polynomials. The following
pretty theorem, a special case of a more general theorem proved by V.~Bergelson,
H.~Furstenberg, and R.~McCutcheon in \cite{BFM}, is a typical result in this
direction.

\begin{thm*} 
Let $\mathcal{F}$ be the collection of all non-empty finite subsets of $\NN$. 
For any polynomial $p \in \polgrp{\Z}{x_1, \dotsc, x_k}$ satisfying $p(0, \dotsc,
0) = 0$, and for any IP-sets $\{n_{\alpha}^{(1)}\}_{\alpha \in \mathcal{F}},
\dotsc, \{n_{\alpha}^{(k)}\}_{\alpha \in \mathcal{F}}$, the set 
\[
	R = \set{p(n_{\alpha}^{(1)}, \dotsc, n_{\alpha}^{(k)})}{\alpha \in \mathcal{F}}
\]
is a set of nice recurrence.
\end{thm*}

To say that $R \subseteq \ZZ$ is a set of \emph{nice recurrence} means that for any
probability space $\msp$, and any invertible measure-preserving transformation $T$ on
$X$, one has
\[
	\limsup_{n \in R} \mu(A \cap T^n A) \geq \mu(A)^2
\]
for all $A \in \mathcal{B}$. Moreover, an \emph{IP-set} is any set of the form
\[
	\Set{n_{\alpha} = \sum_{i \in \alpha} n_i}{\text{$\alpha \in \mathcal{F}$}},
\]
for positive integers $n_0, n_1, n_2, \dotsc$.

As in Furstenberg's result, this inequality has immediate combinatorial applications.
It also turned out that the above theorem was only a first step; much stronger
results---combining IP-convergence, multiple recurrence as in Szemer\'edi's theorem,
and polynomial sets---have since been established, for instance in \cite{BM}.

The purpose of the present paper is to give a different proof for the central result
of \cite{BFM}, using \emph{idempotent ultrafilters} instead of IP-limits. While this
approach is less constructive, it has the advantage of ``making the statements and
proofs cleaner and more algebraic,'' in the words of the survey paper \cite{BSur}.
It also follows the general philosophy that for each result about IP-sets, there
should be an analogous result about idempotent ultrafilters.

The main theorem and its proof are presented in Section~\ref{sec:V}; however, a
better point to begin reading is probably Section~\ref{sec:II}, which treats a
special but typical case, and explains the method of proof in some detail.
Section~\ref{sec:VI} contains a small number of applications, of the type mentioned
above. 

Since ultrafilters on groups and semigroups are used throughout the paper, their
basic properties are reviewed in Section~\ref{sec:I}; readers who are  already
familiar with $\beta\N$, for instance from \cite{BSur}, will recognize all the
material, despite the more general context. To keep the paper self-contained, several
generally known results about operators and integer-valued polynomials have also been
included; these make up Sections~\ref{sec:III} and \ref{sec:IV}.

\subsection*{Note} 

Vitaly Bergelson, who advised me during my first two years in graduate school,
suggested the problem of reproving the results in \cite{BFM} using idempotent
ultrafilters. I am very grateful to him for his help, as well as for countless
pleasant conversations. Unlike wine, the paper has failed to mature during the
several years that it has been stored on the hard drive of my computer; nevertheless,
I have decided to make it available, since it is in my opinion a nice application of
idempotent ultrafilters to recurrence results.

\section{The Stone-\v Cech compactification of a discrete semigroup} \label{sec:I} 

\subsection*{Ultrafilters} 

We begin by reviewing the definition and several basic properties of the space of
ultrafilters. Let $(S, \circ)$ be a commutative semigroup. An \define{ultrafilter} on $S$ is a
collection $p$ of subsets of $S$ with the following four properties:
\begin{enumerate}
\item $S \in p$ and $\emptyset \not\in p$.
\item If $A \in p$, and $B \supseteq A$, then $B \in p$.
\item If $A,B \in p$, then $A \cap B \in p$.
\item For every $A \subseteq S$, either $A \in p$, or $S \setminus A \in p$.
\end{enumerate}
For every $s \in S$, there is a \define{principal} or \emph{trivial} ultrafilter
consisting of all subsets containing $s$; the construction of other ultrafilters
requires the Axiom of Choice.

The space $\beta S$ of all ultrafilters on $S$, suitably topologized, is the Stone-\v
Cech compactification of the discrete space $S$.  After briefly stating the basic
properties of $\beta S$, we will consider two examples: one where $S$ is the group
$\Zn$, and a second one where $S$ equals $\mathcal{F}$, the set of nonempty finite
subsets of $\N$. A good and very comprehensive reference for this topic is the book
by Hindman and Strauss \cite{HiS}.  
 
\subsection*{Terminology} 

Since ultrafilters are collections of sets, the following terminology is convenient
when dealing with their members.  If $p$ is an ultrafilter on $S$, we call a set
\define{$p$-big} if it is contained in $p$; we shall also use the phrase `\define{for
$p$-many $s$}' to mean `for all $s$ in some $p$-big set.' In the case of several
variables, we shall say that $\langle \text{statement} \rangle$ holds `\define{for
$p$-many $s_1, \dotsc, s_n$}' if 
\[
	\set{s_1 \in S}{\set{s_2 \in S}{ \cdots \set{s_n \in S}{\langle \text{statement}
		\rangle} \in p \cdots } \in p} \in p.
\] 
In other words, there should be $p$-many $s_1$, for which there are $p$-many $s_2$,
for which \dots, for which there are $p$-many $s_n$, for which $\langle \text{statement}
\rangle$ is true. Nested sets of exactly this form will play a role during the proof
of the main theorem in Section~\ref{sec:V}.

\subsection*{Basic properties}
	
As was said above, we let $\beta S$ be the set of ultrafilters on $S$, and consider $S$
as a subset of $\beta S$, by identifying an element of $S$ with the principal
ultrafilter it generates.  One can put a topology on $\beta S$, in which the sets
\begin{align*}
	\clA = \set{p \in \beta S}{A \in p} 		&& \text{(for $A \subseteq S$)}
\end{align*}
give a basis for the closed sets; each $\clA$ is both closed and open.
The result is a compact space (this includes the Hausdorff property) that has $S$ as
a discrete and dense subspace.

The semigroup operation $\circ$ extends to $\beta S$; given $p$ and $q$ in $\beta S$, 
their product $p \circ q$ may be defined by the property that for any $A \subseteq S$, 
\[ \label{eq:circ}
	A \in p \circ q  \quad \Longleftrightarrow \quad \set{s \in S}{\set{t \in S}{s \circ t 
	\in A} \in q} \in p.
\]
The new operation is associative and continuous from the left (meaning that for any
$q$, the map $p \mapsto p \circ q$ is continuous), and makes $\beta S$ into a compact
left-topological semigroup. 

\subsection*{Idempotent ultrafilters} 

An ultrafilter $p \in \beta S$ is called \define{idempotent} if it satisfies the
relation $p \circ p = p$. Idempotent ultrafilters are closely related to
\define{IP-sets}, which are sets of the form
\begin{align*}
	\Set{\prod_{i \in \alpha} s_i}{\text{$\alpha \subset \N$ finite, nonempty}},
\end{align*}
for a given sequence $\seq{s_i}{i \in \N}$. Any member of an idempotent ultrafilter
contains an IP-set, and conversely, every IP-set is contained in some idempotent
ultrafilter. This fact is sometimes called Hindman's theorem (see
\cite{BSur}*{Theorem~3.4} for details); it implies that one can find many idempotent
ultrafilters (provided, as usual, that the Axiom of Choice is assumed).

When a finite sequence $s_1,\dotsc, s_n$ is used in place of an infinite one, we shall
denote the resulting \define{finite IP-set} by $\IP{s_1, \dotsc, s_n}$. 
The proof that any member of an idempotent has to contain an IP-set allows a much 
stronger conclusion if we are only looking for finite IP-sets.

\begin{lem} \label{lem:FPsets}
Let $p \in \beta S$ be an idempotent ultrafilter. If $A$ is a $p$-big set, then for 
any $n \in \N$ one has
\[
	\set{s_1 \in A}{ \cdots \set{s_n \in A}{\IP{s_1, \dotsc, s_n} \subseteq A} \in p
		\cdots } \in p.
\] 
In the terminology introduced above, one can say that there are $p$-many $s_1,
\dotsc, s_n$ in $A$ such that $\IP{s_1, \dotsc, s_n} \subseteq A$.
\end{lem}

\subsection*{Limits along ultrafilters} 

Another useful notion is that of a
\define{$p$-limit}, or a limit along some ultrafilter. Let $p \in \beta S$ be an
ultrafilter. Given a map $f \colon S \to Y$ into some topological space $Y$,
we say that a point $y$ is a limit of $f$ along $p$, written
\[
	y = \plim{p}_s f(s),
\] 
if for every neighborhood $U$ of $y$, the set $f^{-1}(U)$ is $p$-big. 
When the target space $Y$ is compact, all $p$-limits exist and are unique.

This notion of limit is related to the Stone-\v Cech compactification in the
following manner. A \define{compactification} of a Hausdorff space $X$ is a compact
space containing $X$ as a dense subspace. The \define{Stone-\v Cech compactification}
$\beta X$ is the universal compactification, in the sense that for any compact space
$Y$ and any continuous map $f \colon X \to Y$, there is one and only one continuous
extension $\phis[f]$ from $\beta X$ to $Y$, as illustrated in the diagram.

\[
\begindc{\commdiag}[25]
\obj(1,3){$X$}
\obj(3,3)[$bX$]{$\smash[b]{\beta}X$}
\obj(1,1){$Y$}
\mor(1,3)(1,1){$\smash{f}$}[\atright, \solidarrow]
\mor(1,3)(3,3){}[\atright, \injectionarrow]
\mor(3,3)(1,1){$\smash[b]{\phis[f]}$}[\atleft, \dasharrow]
\enddc
\]

Every other compactification is a quotient of $\beta X$; furthermore, if $g \colon Y
\to Z$ is a second continuous map of compact spaces, one has $\phis[(fg)] = \phis[f]
\phis[g]$ because of the uniqueness statement.

Now the space $\beta S$, as defined above, is the Stone-\v Cech compactification
of the discrete topological space $S$; 
given any map $f \colon S \to Y$ into a compact space $Y$, the required extension
$\phis[f] \colon \beta S \to Y$ is given by
\begin{align*}
	\phis[f](p) = \plim{p}_s f(s)	 &&\text{(for $p \in \beta S$),}
\end{align*}
which is continous as a map from $\beta S$ to $Y$.

The following lemma is an immediate consequence of the universal property.

\begin{lem}  \label{lem:phistar}
Any map $\phi \colon S \to T$ between two semigroups $S$ and $T$ induces a
continuous map $\phis \colon \beta S \to \beta T$, given by $\phis(p) = \plim{p}_s \phi(s)$. 
A set $B$ is $\phis(p)$-big if, and only if, its preimage $\phi^{-1}(B)$ is $p$-big. 
For any map $f : T \to Y$ into a compact space $Y$, one has
\[
	\plim{p}_s f(\phi(s)) = \plim{\phis(p)}_t f(t).
\]
If $\phi$ is multiplicative, so is $\phis$; in particular, $\phis(p)$
is then always idempotent for idempotent $p \in \beta S$.
\end{lem}

There is another important property of $p$-limits, especially useful for our
purposes.

\begin{lem} \label{lem:plimits}
Let $p$ and $q$ be two elements of $\beta S$. The equality
\[
	\plim{p}_s \plim{q}_t f(s \circ t) = \plim{(p \circ q)}_s f(s)
\]
holds for any map $f \colon S \to Y$ into a compact space $Y$. 
\end{lem}

\begin{proof}
Let $y = \plim{p}_s \plim{q}_t f(s \circ t)$; for any neighborhood $U$ of $y$, the
set
\[
	\set{s}{\plim{q}_t f(s \circ t) \in U}
\]
is $p$-big. Equivalently,  
\[
	\set{s}{\set{t}{f(s \circ t) \in U} \in q} \in p,
\]
and this is nothing but the condition $f^{-1}(U) \in p \circ q$. It follows that the
right-hand limit $\plim{(p \circ q)}_s f(s)$ also equals $y$.
\end{proof}

The lemma explains one useful aspect of idempotent ultrafilters---if $p$ is an
idempotent, one has
\begin{equation} \label{eq:doubleplimit}
	\plim{p}_s \plim{p}_t f(s \circ t) = \plim{p}_s f(s),
\end{equation}
and this relation is at the base of all applications of ultrafilters to recurrence
results.

As an application, let us prove a lemma known as \define{van der Corput's trick}, for
$p$-limits. It provides a useful sufficient condition for a weak $p$-limit in a
Hilbert space to be zero.

\begin{lem} \label{lem:VanDerCorput}
Let $Y$ be a closed ball in a Hilbert space $\Hil$, endowed with the weak topology
(and thus compact). Given a map $f \colon S \to Y$ and an idempotent $p \in \beta S$,
let $y = \plim{p}_s f(s)$. If
\[
	\plim{p}_s \plim{p}_t \inner{f(s \circ t)}{f(t)} = 0,
\]
then $y = 0$.
\end{lem}
\begin{proof}
One uses \eqref{eq:doubleplimit} in a clever way. Notice that for any $N \in
\N$, 
\[
	y = \plim{p}_{s_1} \cdots \plim{p}_{s_N} \frac{1}{N} \sum_{n = 1}^N f(s_n \circ 
	\dotsb \circ s_N).
\]
Using weak lower semi-continuity of the norm, we obtain
\begin{align*}
	\norm{y}^2 &\le 
		\plim{p}_{s_1} \cdots \plim{p}_{s_N} \frac{1}{N^2} \Norm{\sum_{n = 1}^N f(s_n 
		\circ \dotsb \circ s_N)}^2 \\
		&= \plim{p}_{s_1} \cdots \plim{p}_{s_N} \frac{1}{N^2} \Inner{\sum_{m = 1}^N f(s_m 
		\circ \dotsb \circ s_N)}{\sum_{n = 1}^N f(s_n \circ \dotsb \circ s_N)} \\
		&= \frac{1}{N^2} \sum_{m, n} \plim{p}_{s_1} \cdots \plim{p}_{s_N} \inner{f(s_m 
		\circ \dotsb \circ s_N)}{f(s_n \circ \dotsb \circ s_N)}, \\
\intertext{and after collapsing the multiple $p$-limits with the help of 
		\eqref{eq:doubleplimit}, this becomes}
	   &= \frac{1}{N^2} \sum_n \plim{p}_s \norm{f(s)}^2 + \frac{2}{N^2} \Re 
		\sum_{m < n} \plim{p}_s \plim{p}_t \inner{f(s \circ t)}{f(t)} \\
		&= \frac{1}{N} \plim{p}_s \norm{f(s)}^2.
\end{align*}
Since $N$ was arbitrary, we see that $y = 0$.
\end{proof}

We shall now discuss two concrete examples of semigroups and their Stone-\v Cech
compactifications, namely $\beta \Zn$ and $\beta \mathcal{F}$.

\subsection*{\texorpdfstring{Abelian groups and $\beta \Zn$}{Abelian groups and
ultrafilters}}

We are going to use vector notation for elements of $\Zn$, such as $\a = (a_1,
\dotsc, a_n)$. Even though $\Zn$ is a group, the space of ultrafilters $\beta \Zn$ is
only a semigroup, because there are in general no inverses for elements. Still,
we can get information about ultrafilters in $\beta \Zn$ from the group structure of
$\Zn$; in particular, we shall investigate the relationship between subgroups and
idempotent ultrafilters.

Every subgroup of $\Zn$ is itself free, of rank between $0$ and $n$. The
first observation is that subgroups of rank $n$ are contained in every idempotent
ultrafilter.

\begin{lem} \label{lem:Lattice}
For every idempotent $p \in \beta \Zn$, all rank $n$ subgroups are $p$-big.	
\end{lem}

\begin{proof}
A subgroup $L$ of rank $n$ necessarily has finite index. As an ultrafilter, $p$
thus has to contain one of the cosets, say $\z + L$, and as an idempotent, it then has to
contain the set
\[
	\set{\a \in \z + L}{\set{\b \in \z + L}{\a + \b \in \z + L} \in p}
\]
as well. In particular, that set is nonempty. The resulting equation $\z + \z \equiv
\z \mod L$ gives $\z + L = L$, and we can conclude that $L$ itself is $p$-big. 
\end{proof}

We now define the \define{dimension} of an ultrafilter $p$, denoted $\dim p$, to be
the smallest possible rank of a $p$-big subgroup of $\Zn$. Since we expect $p$-big
sets to be large (especially when $p$ is an idempotent), it would be nice if the
dimension of an ultrafilter in $\beta \Zn$ was always $n$. This is not true; for
instance, the principal ultrafilter generated by $0$ is idempotent, and has dimension
zero. But as the following lemma shows, in all such examples, the ultrafilter in
question really lives on a smaller group.

\begin{lem} \label{lem:Dimension}
Let $p \in \beta\Zn$ be an ultrafilter, of dimension $s \in \{0, \dotsc, n\}$.
If $G \subseteq \Zn$ is an arbitrary $p$-big subgroup of rank $s$, then there
is an injective group homomorphism $\phi \colon \Zn[s] \to \Zn$ with image $G$,
and an $s$-dimensional ultrafilter $q \in \beta \Zn[s]$, such that $p = \phis(q)$. If
$p$ is idempotent, then any such $q$ is idempotent as well.
\end{lem}

\begin{proof}
Let $G$ be a $p$-big subgroup of rank $s$ in $\Zn$. Since it is free, it is
isomorphic to $\Zn[s]$, and so there is an injective group homomorphism $\phi \colon
\Zn[s] \to \Zn$ whose image is exactly $G$. Define
\[
	q = \set{A \subseteq \Zn[s]}{\phi(A) \in p};
\]
since $G$ is $p$-big, it is easily verified that $q \in \beta \Zn[s]$, and that
$\phis(q) = p$. Now $q$ has to have dimension $s$, for otherwise $\Zn[s]$, and
therefore also $G$, would contain a $p$-big subgroup of smaller rank,
contradicting the choice of $s$.

Now assume that $p$ is an idempotent ultrafilter. Since $\phis$ is a homomorphism, we
get
\[
	\phis(q \circ q) = \phis(q) \circ \phis(q) = p \circ p = p;
\]
but as $q$ is obviously uniquely determined by the condition that $\phis(q) = p$, it
follows that $q \circ q = q$, and so $q$ is idempotent as well.
\end{proof}

\subsection*{\texorpdfstring{IP-sets and $\beta \mathcal{F}$}{IP-sets and ultrafilters}}

Our second example is the Stone-\v Cech compactification of $\mathcal{F}$, the set of
finite nonempty subsets of $\N$. For any two such finite sets $\alpha$ and $\beta$,
we may form their union $\alpha \cup \beta$; this operation makes $\mathcal{F}$ into
a commutative semigroup. An ultrafilter in this setting is now a set of sets of
finite subsets of $\N$; to avoid confusion, we shall reserve the letters $\alpha,
\beta, \gamma$ for points of $\mathcal{F}$.  We also continue to write $\circ$ for
the semigroup operation on $\beta \mathcal{F}$. The character of this operation is
utterly different from addition on $\beta \Z$; for instance, any principal
ultrafilter is now idempotent.

We are mostly going to look at IP-sets in $\N$ and $\mathcal{F}$ from the point
of view of $\beta \mathcal{F}$. Since the semigroup operation on $\NN$ is addition,
an IP-set is now a set of the form
\[
	\Set{\sum_{i \in \alpha} n_i}{\text{$\alpha \in \mathcal{F}$}},
\]
where $\seq{n_i}{i \in \NN}$ is a sequence of positive integers. This can also be
considered as a map
\[
	\nbull \colon \mathcal{F} \to \N, \qquad \alpha \mapsto n_{\alpha} 
		= \sum_{i \in \alpha} n_i
\]
that satisfies
\begin{equation} \label{eq:nadditive}
	n_{\alpha \cup \beta} = n_{\alpha} + n_{\beta}
\end{equation}
for disjoint $\alpha, \beta \in \mathcal{F}$. It induces a map $\phis[n]$ from $\beta
\mathcal{F}$ to $\beta \N$ that we would like to be structure-preserving, in
particular with regard to idempotents, but this cannot be true. The problem is that
the map $\nbull$ fails to be additive because \eqref{eq:nadditive} holds for
disjoint sets only. One answer is to look at a subclass of ultrafilters in $\beta
\mathcal{F}$, excluding---among other things---the principal ones.

One quickly sees that in order to make use of \eqref{eq:nadditive}, it has to be
possible, when choosing $\beta$ inside a member of some ultrafilter, to make it disjoint
from a given $\alpha$. To accomplish this, we introduce the following notion. We let
\begin{align*}
	C_n = \set{\alpha \in \mathcal{F}}{n \in \alpha}		&& \text{(for $n \in \N$),}
\end{align*}
and call an ultrafilter \emph{congested} if it contains one of the $C_n$, or \emph{uncongested}
if it contains none. Certainly, every principal ultrafilter is congested. 
We also introduce the notation $\alpha < \beta$ to express that the maximum of the 
finite set $\alpha$ is less than the minimum of $\beta$. 

Whenever $p$ is an uncongested ultrafilter and $\alpha \in \mathcal{F}$, the set
\[
	\set{\beta \in \mathcal{F}}{\alpha < \beta}
\]
is obviously $p$-big, being an intersection of complements of certain $C_n$. This means
that for any set $A \in p$, the set of $\beta \in A$ with $\alpha < \beta$ is still
$p$-big, and so we can impose the even stronger condition $\alpha < \beta$ when
choosing an element $\beta$ from $A$. 

Now let us see what the set $\betaFun$ of all uncongested ultrafilters looks like.
An \emph{IP-ring} is a special type of IP-set in $\mathcal{F}$; it consists of an 
infinite sequence $\seq{\alpha_i}{i \in \N}$ of elements of $\mathcal{F}$ satisfying 
$\alpha_0 < \alpha_1 < \alpha_2 < \dotsb$, together with all possible finite unions of these. 
The notation
\[
	\FU{\alpha_i}{i \in \NN}
\]
will be used for such IP-rings. We then have the following result about $\betaFun$ and its
connection with IP-rings.

\begin{lem} \label{lem:Uncongested}
$\betaFun$ is a closed (hence compact) sub-semigroup of $\beta \mathcal{F}$.
Every idempotent $p \in \betaFun$ has the property that if a set is $p$-big, it 
contains an IP-ring.
Conversely, every IP-ring is a member of some uncongested idempotent.
\end{lem}

\begin{proof}
By definition,
\[
	\betaFun = \bigcap_{n \in \N} \cl{\mathcal{F} \setminus C_n}
\]
is an intersection of closed sets, hence closed. Let us show that it is also a
semigroup. If a product $p \circ q$ is congested, it has to contain $C_n$ for some
$n$, and so
\[
	\set{\alpha}{\set{\beta}{n \in \alpha \cup \beta} \in q} \in p.
\]  
If $C_n$ is not in $p$, one of the $\alpha$ in the outer set does not contain $n$, in 
which case the inner $q$-big set equals $C_n$. Either way, one
of the two factors is congested; products of uncongested ultrafilters are therefore 
uncongested. It follows that $\betaFun$ is a compact semigroup.

To verify the second statement---existence of IP-rings in $p$-big sets for idempotent
$p$---the same proof as for Hindman's theorem will work; when choosing elements, one 
simply follows the recipe mentioned above. 

Third, let us show that every IP-ring is contained in some uncongested idempotent
(this also proves the existence of uncongested idempotents). Let $\FU{\alpha_i}{i \in \NN}$
be an IP-ring and set $A_n = \FU{\alpha_i}{i \geq n}$, with obvious
meaning. Following the usual procedure, we will show that the intersection
\[
	\mathcal{A} = \bigcap_{n \in \N} \cl{A_n}
\] 
is a closed, nonempty subsemigroup of $\betaFun$; by Ellis' theorem it then has
idempotents, and any such idempotent is uncongested and contains our IP-ring.

$\mathcal{A}$ is certainly closed and nonempty (use the finite intersection property of
the compact space $\beta \mathcal{F}$). To verify that it is a semigroup, we need to 
check  that for $p, q \in \mathcal{A}$ and any $n \in \N$, the set $A_n$ is a member of
$p \circ q$. 
Given $\alpha \in A_n$, there is some $k > n$ for which $\alpha \cup \beta \in A_n$
for every $\beta \in A_k$. Hence
\[
	\set{\beta}{\alpha \cup \beta \in A_n}
\]
is a $q$-big set for every $\alpha \in A_n$; as a consequence, we have
\begin{align*}
	\set{\alpha}{\set{\beta}{\alpha \cup \beta \in A_n} \in q} \in p 
		&& \text{(for $n \in \N$),}
\end{align*}
and thus $p \circ q \in \mathcal{A}$. Finally, every $p \in \mathcal{A}$ has to be 
uncongested, for $A_{n+1}$ and $C_n$ are always disjoint. 
\end{proof}

One conclusion is that uncongested idempotents do exist; more importantly, they
naturally arise when one is looking at IP-rings in terms of ultrafilters. Indeed, the
lemma is the exact analogon of Hindman's theorem for the case of IP-rings. 

Another useful property of uncongested ultrafilters is stated in the last lemma of
this section; it closely follows our thoughts after \eqref{eq:nadditive}.

\begin{lem} \label{lem:uncongestedidp}
For any IP-set $\nbull \colon \mathcal{F} \to \NN$, the induced map
$\phis[n] \colon \betaFun \to \beta\NN$ is a homomorphism of semigroups. In
particular, $\phis[n](p)$ is idempotent for each uncongested idempotent $p \in
\betaFun$.
\end{lem}

\begin{proof}
Let $p, q \in \betaFun$ be arbitrary uncongested ultrafilters; we need to show that 
\[
	\phis[n](p \circ q) = \phis[n](p) \circ \phis[n](q).
\]
For any subset $A \subseteq \NN$, we have $A \in \phis[n](p)$ exactly when
$\set{\alpha}{n_{\alpha} \in A} \in p$; referring back to the definition of the
operation $\circ$ on page~\pageref{eq:circ}, we then find that
\begin{align*}
	A \in \phis[n](p \circ q) \quad & \Longleftrightarrow \quad
		\set{\alpha}{\set{\beta}{n_{\alpha \cup \beta} \in A} \in q} \in p,
\intertext{while}
	A \in \phis[n](p) \circ \phis[n](q) \quad & \Longleftrightarrow \quad
		\set{\alpha}{\set{\beta}{n_{\alpha} + n_{\beta} \in A} \in q} \in p.
\end{align*}
But since $q$ is uncongested, these two conditions are actually equivalent. Indeed,
given $\alpha \in \mathcal{F}$, we have $n_{\alpha \cup \beta} = n_{\alpha} +
n_{\beta}$ whenever $\beta > \alpha$, and so
\begin{align*}
	\set{\beta}{n_{\alpha \cup \beta} \in A} \cap \set{\beta}{\beta > \alpha}
	&= \set{\beta}{\text{$n_{\alpha \cup \beta} \in A$ and $\beta > \alpha$}} \\
	&= \set{\beta}{\text{$n_{\alpha} + n_{\beta} \in A$ and $\beta > \alpha$}} \\
	&= \set{\beta}{n_{\alpha} + n_{\beta} \in A} \cap \set{\beta}{\beta > \alpha}.
\end{align*}
Now $q$ always contains the set $\set{\beta}{\beta > \alpha}$, and so
\[
	\set{\beta}{n_{\alpha \cup \beta} \in A} \in q \quad \Longleftrightarrow \quad
		\set{\beta}{n_{\alpha} + n_{\beta} \in A} \in q.
\]
This shows that $A \in \phis[n](p \circ q)$ if, and only if, $A \in \phis[n](p) \circ
\phis[n](q)$, and thus proves the lemma.
\end{proof}

This result will later allow us to transfer results from $\beta\N$ or $\beta\Zn$ to
the space $\beta\mathcal{F}$. We shall see applications of this idea in
Section~\ref{sec:V}, after we have proved the main theorem.

\section{An extended example} \label{sec:II}

In this section, we want to give an in-depth discussion of a special case of the main
results, Theorem~\ref{thm:MainA} and Theorem~\ref{thm:MainB}.  We hope that this will
help the reader understand the character of the argument---in particular, how the
induction used in the proof works. We are going to consider the following theorem.

\begin{thmsec} \label{thm:A}
Let $U$ be an arbitrary unitary operator on a Hilbert space $\Hil$, and let 
$\mbull, \nbull \colon \mathcal{F} \to \N$ be any two IP-sets.  If $p \in \beta
\mathcal{F}$ is any uncongested idempotent, the operator $P$ defined by the weak
operator limit 
\[
	P = \plim{p}_{\alpha} U^{m_{\alpha} n_{\alpha}}
\]
is an orthogonal projection. 
\end{thmsec}

The given (weak operator) limit abbreviates the equality
\[
	\inner{P x}{y} = \plim{p}_{\alpha} \inner{U^{m_{\alpha} n_{\alpha}} x}{y}
\]
for all $x, y \in \Hil$.

We are not going to prove this directly, because the presence of the two
IP-sets is inconvenient, in that it obscures part of the underlying structure.
Suppose, for example, that the two IP-sets were (more or less) equal; then the
essentially two-dimensional situation of the theorem would collapse down to a
one-dimensional one, and surely something in the proof will have to change, too. The
problem, in other words, is that there appears to be a notion of \emph{dimension}
behind the theorem---but it is cumbersome to deal with dimension for IP-sets.

On the other hand, as shown by Lemma~\ref{lem:Dimension}, there is a good definition
of dimension for idempotent ultrafilters in $\beta\Zn[2]$. Instead of trying to prove
Theorem~\ref{thm:A} in its present form, we should pass instead to the group
$\Zn[2]$, where we can talk about the rank of subgroups and the dimension of
ultrafilters. 

To this end, define a map $\phi \colon \mathcal{F} \to \Zn[2]$ by 
\begin{align*}
	\phi(\alpha) = \bigl( m_{\alpha}, n_{\alpha} \bigr)		
		&& \text{(for $\alpha \in \mathcal{F}$).}
\end{align*}
Because $p$ is uncongested, the new ultrafilter $q = \phis(p)$ is an idempotent in 
$\beta \Zn[2]$ by Lemma~\ref{lem:uncongestedidp}; moreover, Lemma~\ref{lem:phistar}
changes the limit defining $P$ into
\[
	P = \plim{q}_{\z} U^{z_1 z_2},
\]
where the notation $\z$ is again used for elements of $\Zn[2]$. 
The following more general statement now suggests itself.

\begin{thmsec} \label{thm:B} 
	Let $U$ be a unitary operator on a Hilbert space $\Hil$.  If $q \in \beta \Zn[2]$ 
	is any idempotent, the operator $P$ defined by the weak operator limit
	\[
		P = \plim{q}_{\z} U^{z_1 z_2}
	\]
	is an orthogonal projection.  
\end{thmsec}

Even in this special case, a proof seems to require two separate steps. We begin by
introducing an auxiliary operator
\[
	Q = \plim{q}_{\a} \plim{q}_{\z} U^{a_1 z_2 + a_2 z_1},
\]
the polynomial in the exponent arising from the original $z_1 z_2$ as
\[
	(a_1 + z_1)(a_2 + z_2) - z_1 z_2 - a_1 a_2.
\]

\subsection*{Step 1}

For the time being, we are going to assume that $Q$ is a projection operator, and use
the splitting of the Hilbert space $\Hil = \ker Q \oplus \im Q$ it induces to prove
Theorem~\ref{thm:B}.  To show that $P$ is an orthogonal projection, we appeal to
Lemma~\ref{lem:Projections}: $P$ is clearly normal, being a limit of unitary
operators, and so all we need to do is prove the relation $P^2 = P$. To help with
that, let us also create, for each $\a \in \Zn[2]$, the operator
\[
	Q_{\a} = \plim{q}_{\z} U^{a_1 z_2 + a_2 z_1}.
\]
Since the polynomials involved are linear in $\z$, the reader will prove without much 
effort that each $Q_{\a}$ is an orthogonal projection; for example, one can use the
identity in \eqref{eq:doubleplimit} to show that $Q_{\a}^2 = Q_{\a}$,
and then apply Lemma~\ref{lem:Projections}. Of course, any two of those operators
commute, since they are all limits of powers of $U$.

The point is that under our assumption on $Q$, the weak operator limits
\[
	Q = \plim{q}_{\a} Q_{\a} = \plim{q}_{\a} \plim{q}_{\z} U^{a_1 z_2 + a_2 z_1}
\]
are actually strong ones, as we shall see. To prove the identity $P^2 = P$, let us
first consider the situation on the space $\ker Q$.  If $x$ satisfies $Q x = 0$, we get
\[
	\plim{q}_{\a} \norm{Q_{\a} x}^2 = \plim{q}_{\a} \inner{Q_{\a} x}{x} = 
	\inner{Q x}{x} = 0, 
\]
and because we have convergence in the norm, we can apply van der Corput's trick 
to show $P x = 0$.  The condition in Lemma~\ref{lem:VanDerCorput},
\begin{align*}
	\plim{q}_{\a} \plim{q}_{\z} \Inner{U^{(a_1 + z_1)(a_2 + z_2)} x}{U^{z_1 z_2} x} 
	&= \plim{q}_{\a} \plim{q}_{\z} \Inner{U^{a_1 z_2 + a_2 z_1} x}{U^{-a_1 a_2} x} \\ 
	&= \plim{q}_{\a} \Inner{Q_{\a} x}{U^{-a_1 a_2} x} = 0,
\end{align*} 
is satisfied, and we conclude that $P x = 0$, hence $P^2 x = P x$.  

Next, let us see what happens if $x \in \im Q$. In this case, $Q x = x$, and 
we can write 
\begin{align*}
	\plim{q}_{\a} &\plim{q}_{\z} \norm{U^{a_2 z_2 + a_2 z_1} x - x}^2 \\
	&= 2 \norm{x}^2 - 2 \cdot \Re \plim{q}_{\a} \plim{q}_{\z} \inner{U^{a_1 z_2 + a_2 
	z_1} x}{x} 
	= 2 \norm{x}^2 - 2 \inner{Qx}{x} = 0,	
\end{align*}
from which it follows that
\begin{equation} \label{eq:no1}
	\plim{q}_{\a} \plim{q}_{\z} \norm{U^{(a_1 + z_1)(a_2 + z_2)} x - U^{a_1 a_2} U^{z_1 
	z_2} x} = 0.
\end{equation}
To obtain $P^2 x = P x$, we make use of the identity in \eqref{eq:doubleplimit} for
double $q$-limits; together with \eqref{eq:no1}, we obtain
\begin{align*}
	P x &= \plim{q}_{\z} U^{z_1 z_2} x = \plim{q}_{\a} \plim{q}_{\z} U^{(a_1 + z_1)(a_2 
	+ z_2)} x \\
	&= \plim{q}_{\a} \plim{q}_{\z} U^{a_1 a_2} U^{z_1 z_2} x = P^2 x.  
\end{align*}

We are therefore able to show that $P$ is an orthogonal projection, provided that $Q$
is one. The device of getting strong from weak convergence is frequently useful, by the
way; it is formalized in Lemma~\ref{lem:StrongConvergence} below. 

\subsection*{Step 2} So far, we have been able to reduce Theorem~\ref{thm:B} to the
proof of the following, simpler result.

\begin{thmsec} \label{thm:C}
Let $U$ be a unitary operator on a Hilbert space $\Hil$.   For any idempotent $q
\in \beta \Zn[2]$, the operator
\[
	Q = \plim{q}_{\a} \plim{q}_{\z} U^{a_1 z_2 + a_2 z_1}
\]
is an orthogonal projection.  
\end{thmsec}

Just as in Step~1, everything hinges on having a good splitting of the underlying
Hilbert space $\Hil$. But which splitting one should use depends on the ultrafilter
$q$, more precisely on its dimension---which could be 0, 1, or 2. We will treat these
as separate cases here; in the proof of the main theorem, we shall of course want a
unified approach.

\subsubsection*{Dimension 0} If $q$ is 0-dimensional, it contains the set
$\bigl\{(0,0)\bigr\}$, and since we can restrict to a $q$-big set when taking limits,
$Q$ is simply the identity operator. So this case is trivial.

\subsubsection*{Dimension 1} In case $\dim q = 1$, we can find a subgroup $\Z \c$
(with $\c \neq 0$) of rank one in $q$.  Accordingly, we will use the splitting $\Hil
= \Hil_1 \oplus \Hil_1^{\bot}$, where
\[
	\Hil_1 = \bigcap_{n \neq 0} \ker Q_{n \c} 
		\qquad \text{and} \qquad
	\Hil_1^{\bot} = \overline{\sum_{n \neq 0} \im Q_{n \c}}.
\]
It is then straightforward to show that $Q$ is orthogonal projection onto
$\Hil_1^{\bot}$.

Indeed, if $x$ is an element of $\Hil_1$, then $Q_{n \c} x = 0$ holds for all
nonzero $n$, and since $\c$ generates a $q$-big subgroup and
$\bigl\{(0,0)\bigr\}$ is not $q$-big, we get
\[
	Q x = \plim{q}_{\a} Q_{\a} x = 0
\]
On the other hand, to show that $Q$ restricted to $\Hil_1^{\bot}$ is the
identity, we need only consider $x \in \im Q_{n \c}$, as the span of
these vectors is dense.  For any such $x$, we have
\[
	Q_{n \c} x = \plim{q}_{\z} U^{n(c_1 z_2 + c_2 z_1)} x = x,
\]
which, as before, can be strengthened to
\begin{equation} \label{eq:no2}
	\plim{q}_{\z} \norm{U^{n(c_1 z_2 + c_2 z_1)} x - x} = 0.
\end{equation}
Now we need to extend this equality, true for only one vector $n \c$, 
to some $q$-big set of vectors. We leave it for the reader to check that \eqref{eq:no2} 
actually gives
\[ 
	\plim{q}_{\z} \norm{U^{N n(c_1 z_2 + c_2 z_1)} x - x} = 0.
\]
for any $N \in \Z$. (Hint: Use a telescoping sum.)
But the set $\Z \cdot n \c$ is again a $q$-big subgroup (since $q$ is idempotent), and so 
\[ 
	\plim{q}_{\a} \plim{q}_{\z} \norm{U^{a_1 z_2 + a_2 z_1} x - x} = 0,
\]
which implies $Q x = x$. So $Q$ is indeed an orthogonal projection, with image $\Hil_1^{\bot}$.

\subsubsection*{Dimension 2} Finally, let us treat the really interesting case of a
two-dimensional $q$. We use the same argument as before, only the splitting has to be
adjusted a bit; instead of focusing on one specific subgroup (like $\Z \c$), we shall
consider all of them. So let
\[
	\Hil_2 = \bigcap_{\a, \b} \ker Q_{\a} Q_{\b}
		\qquad \text{and} \qquad
	\Hil_2^{\bot} = \overline{\sum_{\a, \b} \im Q_{\a} Q_{\b}}
\]
be the two complementary subspaces, where both the intersection and the sum are
taken over those $\a, \b \in \Zn[2]$ for which the subgroup $\Z \a + \Z \b$ has
rank two. Again, it will turn out that $Q$ is orthogonal projection onto
$\Hil_2^{\bot}$.

To prove that $Q$ fixes every vector in $\Hil_2^{\bot}$, we may again limit our
attention to elements $x \in \im Q_{\a} Q_{\b}$ for two vectors $\a$ and $\b$
that span a rank two subgroup. The same argument as before shows that
\[
	\plim{q}_{\z} \norm{U^{M (a_1 z_2 + a_2 z_1) + N (b_1 z_2 + b_2 z_1)} x - x} = 0
\]
for any $M, N \in \Z$; the group generated by $\a$ and $\b$ is $q$-big
(remember that it contains some lattice), and so we have
\[
	\plim{q}_{\z} \norm{U^{c_1 z_2 + c_2 z_1} x - x} = 0
\]
for $q$-many $\c \in \Zn[2]$. Taking the $q$-limit over $\c$ then gives the
result, namely that $Q x = x$.

To finish the proof, we have to deal with an arbitrary $x \in \Hil_2$ and show that $Q x = 0$.
What we know is that $Q_{\a} Q_{\b} x = 0$ for any two vectors $\a$ and $\b$
with a two-dimensional span. This is a lot of information, since there are many
such pairs---in fact, for any nonzero $\a \in \Zn[2]$, a $q$-big set of $\b$
has the required property. For suppose, to the contrary, that $q$-many vectors
$\b$ could span only a subgroup of rank one together with $\a$. As $\Z \a + \Z
\b$ is of rank one if and only if $\b$ is a multiple of $\a/g$ (here $g$ is the
greatest common divisor of the components of $\a$), it would follow that $\Z
\a/g$ was a $q$-big subgroup of $\Zn[2]$, contradicting our assumption on the
dimension of $q$.

In particular, we know $Q_{\a} Q_{\b} x = 0$ for sufficiently many $\a$ and $\b$ to
conclude that 
\[
	Q^2 x = \plim{q}_{\a} \plim{q}_{\b} Q_{\a} Q_{\b} x = 0;
\]
but now the operator $Q$ is very evidently self-adjoint and so $Q x = 0$ as
well. This shows that $Q$ is a projection and ends the proof of
Theorem~\ref{thm:C}.

\subsection*{Conclusions}

Let us end this section with several remarks concerning the nature of the
proof. Firstly, the reader will have observed the balance---crude in the case of
one-dimensional $p$, slightly more subtle for two dimensions---between the two
spaces of the splitting. In the first space, $\Hil_1$ or $\Hil_2$, where we use
the null spaces of projections, we need to intersect a large number of them to
make up for the weakness of each individual piece; for each operator, we only know
that one particular $p$-limit is zero, and that amounts to nothing by itself. 
On the other hand, the orthogonal complements, $\Hil_1^{\bot}$ or
$\Hil_2^{\bot}$, involve image spaces of projections; the knowledge that we
gain from each piece is far stronger here, and so we can afford to have this
knowledge in only one case. 

Secondly, it is clear that the dimension of the ultrafilter is important. It was
pointed out before that, although the same concept is lurking around in
Theorem~\ref{thm:A}, it is less easily quantified and dealt with there. The passage
from $\beta\mathcal{F}$ to $\beta\Zn[2]$ helps to make it visible, by removing the
IP-sets.  Moreover, it is of course unnecessary to handle the various dimensions by
different arguments; the proof is really the same in all cases. Indeed, in
Section~\ref{sec:V}, when proving the main theorem, the first step will be to adjust
the dimension of the surrounding group to make it match that of the ultrafilter $p$.
This is where Lemma~\ref{lem:Dimension} will play its part.

Finally, the more general result in the main theorem requires more effort to prove;
although the proof is, in essence, the same as the one given here, there are several
technical points that need to be dealt with. In particular, the presence of
polynomials of higher degree needs special care. The following two sections contain a
few tools that will be helpful; all necessary results about polynomials are collected
in Section~\ref{sec:IV}.

\section{Orthogonal projections and limits} \label{sec:III}

In this section, we prove two simple but useful results about orthogonal projections
and limits; these are well-known, of course. The first, which has already been used,
gives a condition for an operator to be a projection.

\begin{lem} \label{lem:Projections}
A normal operator $P$ on a Hilbert space $\Hil$ is an orthogonal projection 
if, and only if, it satisfies $P^2 = P$.
\end{lem}

\begin{proof} 
Necessity is clear. If $P$ meets the condition, the product $Q = P^{\ast} P$ of
$P$ and its adjoint also does. The latter is self-adjoint in addition, hence
satisfies $\inner{Q x}{x - Q x} = \inner{x}{Q (x - Q x)} = 0$ and is therefore an 
orthogonal projection onto the image space of $Q$. For $x \in \ker Q$, one has 
$\norm{P x}^2 = \inner{x}{Q x} = 0$; for $x \in \im Q$, one has $P x = P Q x = 
P P^{\ast} P x = P^{\ast} P x = x$. Consequently, $P = Q$ is an orthogonal
projection. 
\end{proof}

Our second lemma deals with the question of when certain `weak' limits in a Hilbert  
space $\Hil$ are `strong' limits and is meant to collect the pieces of reasoning used in the 
previous section. The whole discussion is somewhat vague but the result
is useful, though nearly self-evident. Let `$\lim$' be an
abbreviation for some unspecified $p$-limit, maybe even a multiple one, and let $I$ be the 
corresponding index set. So for example, $\lim$ might equal $\plim{p}_{\a} \plim{p}_{\b}$,
with both $\a$ and $\b$ ranging over $\Zn[2]$, in which case the index set $I$
would be $\Zn[2] \times \Zn[2]$.

By what we said in Section~\ref{sec:I}, the limit $\lim x_i$ is defined, in the weak
topology, for every bounded family $\seq{x_i}{i \in I}$ of points, as any closed ball
in $\Hil$ is weakly compact. $\lim x_i = x$ thus means that for any $y \in \Hil$,
\[
	\lim \inner{x_i}{y} = \inner{x}{y}.
\]
On the other hand, the convergence is called \emph{strong} if
\[
	\lim \norm{x_i - x} = 0. 
\]
The norm is weakly lower semi-continuous---if $x = \lim x_i$, then 
\[
	\norm{x} \le \lim \norm{x_i}.
\]

One can also define the notions of weak and strong \emph{operator limits}; in fact, we have 
already been using these. We say, for example, that $T$ is the weak operator limit of
a family $\seq{T_i}{i \in I}$ of operators---and write $T = \lim T_i$---if
\begin{align*}
	\lim \inner{T_i x}{y} = \inner{T x}{y}			&& \text{(for $x, y \in \Hil$)}.
\end{align*}
A few simple calculations then give the following result.

\begin{lem} \label{lem:StrongConvergence}
Let $\lim$ and $I$ be defined as above.
\begin{enumerate}
\item If $T = \lim U_i$ is the weak operator limit of a family $\seq{U_i}{i \in I}$ of unitary 
	operators, then $T$ is normal. For $x \in \Hil$, one has $\lim 
	\norm{T x - U_i x} = 0$ if, and only if, $\norm{T x} = \norm{x}$. In case $T$ is an
	orthogonal projection, this happens exactly when $T x = x$.
\item If $T = \lim P_i$ is the weak operator limit of a family $\seq{P_i}{i \in I}$ of 
	orthogonal projections, then $T$ is self-adjoint. For $x \in \Hil$,
	one has $\lim \norm{T x - P_i x} = 0$ if, and only if, $\inner{T x}{x - T x} = 0$. 
	In case $T$ is itself an orthogonal projection, this condition is always satisfied.   
\end{enumerate}
\end{lem}

The result looks innocent enough, but it will be used frequently.

\section{Polynomials} \label{sec:IV}

We shall be using polynomials in several variables for which the following notation 
seems appropriate. Lower-case Roman letters with arrows will usually denote 
$n$-dimensional vectors, e.g. $\z = (z_1, \dotsc, z_n)$. We shall be speaking of 
polynomials \emph{in the variable $\z$}, meaning really polynomials in the $n$ 
variables $z_1, \dotsc, z_n$. The \emph{degree} of such a polynomial will be its 
total degree. We shall also consider polynomials in several multi-dimensional 
variables: $f(\a, \c)$, say, would be a polynomial in both sets of variables; the 
\emph{degree in $\a$} is the total degree of $f$ as a polynomial in $a_1, \dotsc, a_n$, 
and so on.

If $G$ is any Abelian group, we shall let $\polgrp{G}{\z[1], \dotsc, \z[s]}$ stand for 
the additive group of polynomials in $\z[1], \dotsc, \z[s]$ with coefficients in $G$; 
we shall ignore the multiplicative structure. For the subgroup of those polynomials
in $\polgrp{\Q}{\z[1], \dotsc, \z[r]}$ that produce integer values for integer 
arguments, we shall write $\polint{\z[1], \dotsc, \z[s]}$.

In the one-dimensional case, $\polint{x}$ consists of all polynomials $f \in
\polgrp{\Q}{x}$ with $f(\Z) \subseteq \Z$. It is a free group with basis consisting
of the polynomials 
\begin{align*}
	\binom{x}{i} = \frac{x (x-1) \dotsm (x - i + 1)}{i!}  	&& \text{(for $i \ge 0$).} 
\end{align*}
Indeed, if $f$ is any polynomial in $\polint{x}$ and $m$ its degree, one may find $m + 1$
integers $a_0, \dotsc, a_m$ such that
\[
	f(x) = \sum_{i = 0}^m a_i \binom{x}{i},
\] 
by evaluating successively at $x = 0, 1, \dotsc, m$, and solving the resulting system of 
equations. For any number $d \ge 0$, the polynomials of degree at most $d$ form a free 
subgroup of rank $d + 1$.

The same argument, applied inductively, proves the following.

\begin{lem} \label{lem:Int} 
	$\polint{\z[1], \dotsc, \z[s]}$ is always a free group; for any integer $d \ge 0$, 
	the polynomials of total degree at most $d$ constitute a free subgroup of finite rank, 
	and so do the polynomials of degree at most $d$ in each variable.  
\end{lem}

We now introduce one more useful notion. In the example in Section~\ref{sec:II}, when 
dealing with
the polynomial $f(x, y) = xy$, we found it useful to form the new polynomial 
\[
	f(a + x, b + y) - f(a, b) - f(x, y),
\] 
essentially because its degree in $(x, y)$ was lower. 
An appropriate generalization is as follows. Given a polynomial $f(\z) \in \polgrp{G}{\z}$ 
and an integer $s \ge 1$, we recursively define a new polynomial $\dsum{s}{f(\z[1], \dotsc, 
\z[s])}$, by letting $\dsum{1}{f(\z[1])} = f(\z[1])$, and 
\begin{multline} \label{eq:DeltaRecursively}
	\dsum{s+1}{f(\z[1], \dotsc, \z[s+1])} =\\ 
	\dsum{s}{f(\z[1], \dotsc, \z[s] + \z[s+1])} - \dsum{s}{f(\z[1], \dotsc, \z[s])} 
	- \dsum{s}{f(\z[1], \dotsc, \z[s+1])}.
\end{multline}
Of course, $\dsum{s}{}$ can be described explicitly as
\[
	\dsum{s}{f(\z[1], \dotsc, \z[s])} = \sum_{\emptyset \neq \alpha \subseteq 
	\{1, \dotsc, s \}} (-1)^{s - \abs{\alpha}} \cdot f \bigl( \sum_{i \in \alpha} \z[i] 
	\bigr),
\]  
and the symmetry in all arguments is more apparent from this description.

Let us investigate some properties of $\dsum{s}{}$. First, we have the following easy
lemma.

\begin{lem} \label{lem:Degree}
The polynomial $\dsum{2}{f(\a, \z)} = f(\a + \z) - f(\a) - f(\z)$ is of lower degree
in each variable than $f(\z)$ itself, whenever $f(\z) \in \polgrp{G}{\z}$ is nonzero.
\end{lem}

Now let $f(\z)$ be of degree $d \ge 1$. Since $\dsum{s}{f(\z[1], \dotsc, \z[s])}$ is 
symmetric in its $s$ arguments, the lemma---together with the relations
\eqref{eq:DeltaRecursively}---immediately shows that its degree in any variable can be at most
$(d + 1 - s)$. It follows that $\dsum{d+1}{f(\z[1], \dotsc, \z[d+1])}$ is a constant,
with value
\[
	\dsum{d+1}{f(0, \dotsc, 0)} = \sum_{k=1}^{d+1} (-1)^{d+1-k} \binom{d+1}{k}
	f(0) = (-1)^d f(0).
\]

If $f$ happens to satisfy $f(0) = 0$, one has $\dsum{d+1}{f(\z[1], \dotsc, \z[d+1])} = 0$.
For reasons of symmetry, $\dsum{d}{f(\z[1], \dotsc, \z[d])}$ is then linear in each of 
its $d$ arguments. 

We have shown the following.

\begin{lem} \label{lem:Delta}
For any polynomial $f(\z) \in \polgrp{G}{\z}$ of degree $d \ge 1$, one has the relation 
\[
	\dsum{d+1}{f(\z[1], \dotsc, \z[d+1])} = (-1)^d f(0).
\] 
If $f(0) = 0$, then $\dsum{d}{f(\z[1], \dotsc, \z[d])}$ is a linear function of each argument.
\end{lem}

A third lemma deals with the case of homogeneous $f$.

\begin{lem} \label{lem:Homogeneous}
Let $f(\z) \in \polgrp{G}{\z}$ be a homogeneous polynomial of degree $d \ge 1$. Then 
\[
	\dsum{d}{f(\a, \dotsc, \a)} = d! f(\a).
\]
\end{lem}

\begin{proof}
Using homogeneity, we have
\[
	\dsum{s}{f(\a, \dotsc, \a)} = \sum_{k=1}^s (-1)^{s-k} \binom{s}{k}
	f(k \a) = \sum_{k=1}^s (-1)^{s-k} \binom{s}{k} k^d \cdot f(\a).
\]
We obviously have to evaluate sums of the form 
\begin{align*}
	C(s,m) = \sum_{k=1}^s (-1)^{s-k} \binom{s}{k} k^m		
		&& \text{(for $s \ge 1, m \ge 0$);}
\end{align*}
in particular, $C(d,d) = d!$ is what we need to show. From the previous lemma, we 
already know that $C(s,m) = 0$ whenever $s > m$. Now we compute
\begin{align*}
	C(m+1,m+1) &= \sum_{k=1}^{m+1} (-1)^{m+1-k} \binom{m+1}{k} k^{m+1} \\
				  &= \sum_{k=1}^{m+1} (-1)^{m-(k-1)} (m+1) \binom{m}{k-1} k^m \\
				  &= (m+1) \sum_{l=0}^{m} (-1)^{m-l} \binom{m}{l} (l+1)^m \\
				  &= (m+1) \sum_{l=0}^{m} (-1)^{m-l} \binom{m}{l} \Bigl( 1 + \sum_{i=1}^m
				  \binom{m}{i} l^i \Bigr) \\
				  &= (m+1) \sum_{l=0}^{m} (-1)^{m-l} \binom{m}{l} + (m+1) \sum_{i=1}^m 
				  \binom{m}{i} C(m,i) \\
				  &= (1-1)^m + (m+1) C(m,m) = (m+1) C(m,m),
\end{align*}
and together with $C(1,1) = 1$ this proves the lemma by induction.
\end{proof}

We will now use the previous results to establish an important technical lemma; it is
essential for the proof of the main theorem in Section~\ref{sec:V}. Note that it
introduces a feature not present in the example of Section~\ref{sec:II}, where we had
to deal with polynomials of no more than first degree. It does, however, fit in with
the general philosophy behind the argument---there is one situation in the proof
where one has to make a lot from apparently nothing, meaning where one has to create
useful $p$-big sets from useless ones, and the following lemma does just that.  

\begin{lem} \label{lem:Key}
Suppose that $p \in \beta\Zn$ is an $n$-dimensional idempotent. Let $G$ 
be an Abelian group and let $v(\z) \in \polgrp{G}{\z\,}$ be a polynomial in 
$\z = (z_1, \dotsc, z_n)$. Fix a subgroup $V \subseteq G$. If the set
\[
	B = \set{\b \in \Zn}{\text{$N \cdot v(\b) \in V$ for some $N \ne 0$}}
\] 
is $p$-big, then so is the set 
\[
	A = \set{\a \in \Zn}{v(\a)-v(0) \in V}.
\]  
\end{lem}

\begin{proof}
The idea of the proof is simple: Whenever $B$ contains $\IP{\b[1], \dotsc, 
\b[s]}$, there is some integer $N \neq 0$ such that $N \cdot \dsum{s}{v(\b[1], \dotsc, 
\b[s])} \in V$. Using this and the previous results, we can extract from $v$ its 
homogeneous parts of different degrees, and show that they are each contained in $V$ 
for $p$-many $\z$. Proceeding stepwise, we shall prove two things:
\begin{enumerate}[label=\Roman{*}., ref=\Roman{*}]
\item Without loss of generality, it may be assumed that $v(0) = 0$. 
\item The set $A = \set{\a}{v(\a) \in V}$ is $p$-big.
\end{enumerate}
The details are as follows.

\subsection*{\texorpdfstring{I}{Step I}} Let $d$ be the degree of $v(\z)$. The set $B$ is a member of the idempotent 
$p$, and by Lemma~\ref{lem:FPsets}, we may select $d + 1$ elements $\b[1], \dotsc,
\b[d+1]$ in $B$ with 
\[
	F = \IP{\b[1], \dotsc, \b[d+1]} \subseteq B.
\]
For each $\b$ in this finite $\mathrm{IP}$-set, there exists some  $N \neq 0$ such that 
$N \cdot v(\b) \in V$; if we let $N_1$ be the greatest common divisor of these
numbers, we guarantee that $N_1 \cdot v(\b) \in V$ for every $\b \in F$. By
Lemma~\ref{lem:Delta}, we now have
\[
	N_1 \cdot v(0) = N_1 \cdot (-1)^d \dsum{d+1}{v(\b[1], \dotsc, \b[d+1])} \in V.
\]
If $\b$ is any element of $B$ and $N \cdot v(\b) \in V$, the expression 
$N_1 N \cdot \bigl( v(\b) - v(0) \bigr)$ is also an element of $V$; this means that 
the set
\[
	\set{\b}{\text{$N \cdot \bigl( v(\b) - v(0) \bigr)\in V$ for some $N \ne 0$}}
\]
is equally $p$-big. We may therefore replace $v(\z)$ by $v(\z) - v(0)$ and assume 
$v(0) = 0$.

\subsection*{\texorpdfstring{II}{Step II}} We decompose $v$ into homogenous polynomials,
\[
	v(\z) = \sum_{i=1}^d h_i(\z),
\]
say, with $h_i(\z)$ homogeneous of degree $i$. For $1 \le i \le d$, let
\[
	A_i = \set{\a}{h_i(\a) \in V}.
\]
We shall argue that $A_d \in p$; once this is known, the same reasoning applies 
to the polynomial $v(\z) - h_d(\z)$ where it gives $A_{d-1} \in p$, and so on, until
one has $A_i \in p$ for each $i$. The result follows because $A$ contains the intersection
of all $A_i$. 

To show that $A_d$ is $p$-big, we use a similar---but more careful---approach as 
before. By Lemma~\ref{lem:FPsets}, there are actually $p$-many $\b[1], \dotsc, \b[d] 
\in B$ with $\IP{\b[1], \dotsc, \b[d]} \subseteq B$, and again, to each choice we may find some 
$N_1 \ne 0$ such that $N_1 \cdot \dsum{d}{v(\b[1], \dotsc, \b[d])}$ is an element of
$V$. When selecting only $\b[1], \dotsc, \b[d-1]$ from these, the set
\[
	B_d = \set{\b}{\text{$N \cdot \dsum{d}{v(\b[1], \dotsc, \b[d-1], \b)} \in V$ for 
	some $N \ne 0$}}
\]
is then always $p$-big by construction.

$B_d$ is also a subgroup of $\Zn$, the polynomial $\dsum{d}{v(\z[1], \dotsc, \z[d])}$
being linear in each variable (see Lemma~\ref{lem:Delta}). Since $p$ is 
$n$-dimensional, this subgroup has to have rank $n$ and has to contain a set of the
form $L \cdot \Zn$ for some nonzero $L$. For each $\b \in \Zn$, one gets
\[
	N L \cdot \dsum{d}{v(\b[1], \dotsc, \b[d-1], \b)} = N \cdot \dsum{d}{v(\b[1], 
	\dotsc, \b[d-1], L \cdot \b)} \in V
\] 
for some $N \neq 0$, and so we conclude that $B_d = \Zn$.

Next, consider the set
\[
	B_{d-1} = \set{\b}{\text{for some $N \neq 0$, $N \cdot \dsum{d}{v(\b[1], \dotsc, 
	\b[d-2], \b, \a[d])} \in V$ for all $\a[d]$}};
\]
by the above, it is $p$-big, and a repetition of the argument shows that $B_{d-1} = 
\Zn$, too. 
Continuing in this way, we eventually find a nonzero integer $N_2$ such that 
\[
	N_2 \cdot \dsum{d}{v(\a[1], \dotsc, \a[d])} \in V 
\]
for any choice of $\a[1], \dotsc, \a[d] \in \Zn$.

By Lemma~\ref{lem:Degree}, we have 
\[
	\dsum{d}{h_d(\a[1], \dotsc, \a[d])} = \dsum{d}{v(\a[1], \dotsc, \a[d])},
\] 
because all terms of degree less than $d$ disappear. Finally, using 
Lemma~\ref{lem:Homogeneous} we get 
\[
	N_2 d! \cdot h_d(\a) = N_2 \cdot \dsum{d}{v(\a, \dotsc, \a)} \in V
\]
for all $\a$, and hence $h_d(\a) \in V$ whenever $\a \in N_2 d! \cdot \Zn$. 
The latter set is $p$-big and so $A_d \in p$. This ends the proof of the second part,
and establishes the lemma.
\end{proof}

\section{Statement and proof of the main results} \label{sec:V}

After all the preliminary work in the previous two sections, we are now ready to
state and prove the main result. The notation is somewhat heavy, but this generality
is needed because of the inductive nature of the proof.

\begin{thmsec} \label{thm:MainA}
For $j=1, \dotsc, s$, let $p_j \in \beta\Zn[n_j]$ be an idempotent, and let
$U_1, \dotsc, U_m$ be commuting unitary operators on a Hilbert space $\Hil$.
Given any $m$ polynomials $f_1, \dotsc, f_m \in \polint{\z[1],
\dots,\z[s]}$---with $\z[j]$ of dimension $n_j$---satisfying $f_i(0) = 0$ for
all $i = 1, \dotsc, m$, define an operator $P$ on $\Hil$ by
\[
	P = \plim{p_1}_{\z[1]} \cdots \plim{p_s}_{\z[s]} \unop{f_i(\z[1], \dotsc, \z[s])}.
\]
Then $P$ is always an orthogonal projection. Any two operators defined in this way commute.
\end{thmsec}

\begin{proof} 
We will suppose that all $n_j$ are equal to some $n$ and all $p_j$ are
equal to some $p$, in order to simplify notation. The argument need not be changed in
any way to accommodate the more general situation---because of the inductive
character of the proof, we find ourselves working on no more than one $p$-limit at a
time anyway.

Let us first observe that the last part of the statement---commutativity of different
projections---is obviously true, for all operators generated for various selections
of polynomials are certain weak limits of commuting unitary operators. We may
therefore assume commutativity wherever needed.

The remainder of the proof is essentially by induction on the number $s$ of
$p$-limits taken, but there are some complications involving the case $s=1$.  In
fact, different arguments are needed for $s = 1$ and for $s \ge 2$, since the outmost
$p$-limit is one of unitary operators in the former situation, but one of projections
in the latter. To make the induction be more transparent, we shall use the
abbreviation $(s,d)$ when referring to the statement of the theorem for a certain
value of $s$ and all possible choices of polynomials $f_i$ of degree at most $d$ in
any of their variables $\z[j]$.

The proof will be divided into several steps, the second and third inductive in
nature:
\begin{enumerate}[label=\Roman{*}., ref=\Roman{*}]
\item We argue that the ultrafilter $p$ may be assumed to be $n$-dimensional, without loss
of generality. 
 
\item \label{item:II} We establish the case $(1,d)$, that is, we show that the operator
\[
	P = \plim{p}_{\z} \unop{f_i(\z)},
\]	 
with $f_i$ of degree at most $d$, is an orthogonal projection, assuming the 
statement of the theorem in the two cases $(1,d-1)$ and $(2,d-1)$. Specifically, we
need to assume that the operators
\[
	Q_{\a} = \plim{p}_{\z} \unop{f_i(\a + \z) - f_i(\a) - f_i(\z)}
\]
and
\[
	Q = \plim{p}_{\a} Q_{\a}
\]
are orthogonal projections.

\item \label{item:III} For $s \ge 2$, we derive $(s,d)$ from $(s-1,d)$ and $(1,d)$. 
Essentially, we introduce the new polynomials
\[
	f_i'(\a, \z[2], \dotsc, \z[s]) = f_i(\a, \z[2], \dotsc, \z[s]) - f_i(\a, 0, \dotsc, 0)
\]
and assume that for each $\a \in \ZZ^n$, the operator
\[
	P_{\a} = \plim{p}_{\z[2]} \cdots \plim{p}_{\z[s]} \unop{f_i'(\a, \z[2], 
	\dotsc, \z[s])}
\]
is an orthogonal projection. We then construct a suitable splitting of the underlying
Hilbert space $\Hil$.

\item Using the splitting introduced in the previous part, we show that $P$ is an orthogonal
projection.
\end{enumerate}

Once we have established all of the previous, our work will be done. For the statement of
the theorem is definitely true in the case $(1,0)$---if all $f_i$ equal zero, $P$ is
just the identity---and then \ref{item:II} and \ref{item:III} suffice to prove the
entire theorem by induction.  Let us now take a detailed look at the four steps.

\subsection*{\texorpdfstring{I}{Step I}} 

Using Lemma~\ref{lem:Dimension}, we begin by adjusting the situation to make sure that
the dimension of $p$ is equal to the rank of the group. Of course, this will change
the polynomials under consideration; but since their degrees are not increased, it
does not affect the proof. To write this down precisely is somewhat cumbersome;
so let us look at the case of just one operator $U$ and one polynomial $f(\z)$ to see
what happens. We shall use $\z$ for an $n$-dimensional and $\w$ for an
$s$-dimensional variable. If $s = \dim p$ and $\phi$ are as in Lemma~\ref{lem:Dimension},
then $q = \phis(p)$ is idempotent and $s$-dimensional. By Lemma~\ref{lem:plimits},
\begin{equation} \label{eq:Transformation}
	\plim{p}_{\z} U^{f(\z)} = \plim{q}_{\w} U^{f(\phi^{-1}(\w))}.
\end{equation}
But $g(\w) = f(\phi^{-1}(\w)) \in \polint{\w}$ satisfies $g(0) =
0$, and is of degree no larger than that of $f$; thus a proof that the right-hand
operator in \eqref{eq:Transformation} is a
projection gives the result for the left-hand one, too. The same is true in the
general setting of the theorem, though somewhat unpleasant to write down in detail. 

In any case, we shall assume from now on that $\dim p = n$.  Lemma~\ref{lem:Key} is
then applicable; it will make its entry in the third step of the proof.

\subsection*{\texorpdfstring{II}{Step II}} 

As stated above, we shall now assume that both $Q$ and all the $Q_{\a}$ are 
orthogonal projections; this is permissible because each polynomial
\[ 
	f_i(\a + \z) - f_i(\a) - f_i(\z)
\] 
has degree at most $(d - 1)$ in $\a$ and $\z$ (see Lemma~\ref{lem:Degree}).
The projection $Q$ induces a splitting $\Hil = \ker Q \oplus \im Q$ of the underlying
Hilbert space; we shall use it to conclude that $P^2 x = P x$ for all $x \in
\Hil$.

First, consider $x \in \ker Q$. Since $Q x = \plim{p}_{\a} Q_{\a} x$, 
Lemma~\ref{lem:StrongConvergence} implies that the convergence is strong,
\[
	\plim{p}_{\a} \norm{Q_{\a} x} = 0.
\] 
We now use van der Corput's trick to get $P x = 0$, the condition
\begin{align*}
	& \plim{p}_{\a} \plim{p}_{\z} \Inner{\unop{f_i(\a + \z)} x}{\unop{f_i(\z)} x} \\
	=& \plim{p}_{\a} \plim{p}_{\z} \Inner{\unop{f_i(\a + \z) - f_i(\a) - f_i(\z)} x}
	{\unop{-f_i(\a)} x} \\
	=& \plim{p}_{\a} \Inner{Q_{\a} x}{\unop{-f_i(\a)} x} = 0 
\end{align*}
in Lemma~\ref{lem:VanDerCorput} being fulfilled. A fortiori, $P^2 x = P x$.

Second, consider an arbitrary $x \in \im Q$, which then satisfies $Q x = x$. We again
get strong convergence from Lemma~\ref{lem:StrongConvergence}, so that
\[
	\plim{p}_{\a} \plim{p}_{\z} \Norm{\unop{f_i(\a + \z)} x - \unop{f_i(\a) + 					f_i(\z)} x} = 0.
\] 
But then
\begin{align*}
	P^2 x &= \plim{p}_{\a} \plim{p}_{\z} \unop{f_i(\a) + f_i(\z)} x 
			 = \plim{p}_{\a} \plim{p}_{\z} \unop{f_i(\a + \z)} x \\
			&= \plim{p}_{\z} \unop{f_i(\z)} x = P x.		 
\end{align*}
We now have $P^2 = P$; obviously, $P$ is normal, and the result---that $P$ is an 
orthogonal projection---follows from Lemma~\ref{lem:Projections}. 

\subsection*{\texorpdfstring{III}{Step III}} 

This is the most interesting part of the argument. We start from the 
inductive assumption that each $P_{\a}$ is an orthogonal projection, and aim for a useful 
splitting of the space $\Hil$, depending on the projections $P_{\a}$ and the 
polynomials 
\[
	f_i'(\a, \z[2], \dotsc, \z[s]) = f_i(\a, \z[2], \dotsc, \z[s]) 
	- f_i(\a, 0, \dotsc, 0).
\] 
Note that $f_i'(\a, 0, \dotsc, 0) = 0$ for any $\a$, which means that the new
polynomials $f_i'$ still satisfy the conditions of the theorem, while having fewer
variables.

After the introduction of $g_i(\a) = f_i(\a, 0, \dotsc, 0)$, the operator $P$ is then
given by the limit
\[
	P = \plim{p}_{\a} \Bigl( \unop{g_i(\a)} \Bigr) P_{\a}.
\]

We shall let $F \subseteq \polgrp{\mathrm{Int}}{\z[2], \dotsc, \z[s]}$ denote the set of
polynomials of degree at most $d$ in each variable; $F$ is a free group of finite rank by
Lemma~\ref{lem:Int}. The product $F^m$ is also free, as are all of its subgroups, and for
any $\a \in \Zn$, the vector
\[
	v(\a) = \bigl( \tilde{f}_1(\a, \z[2], \dotsc, \z[s]), \dotsc, \tilde{f}_m(\a, \z[2], 
	\dotsc, \z[s]) \bigr)
\]
is an element of $F^m$. We introduce the notation $V(\a[1], \dotsc, \a[r])$ for the
subgroup of $F^m$ generated by the vectors $v(\a[1]), \dotsc, v(\a[r])$.

The \emph{crucial idea} is to let $r \ge 0$ be the maximal integer for which 
\[
	\set{\a[1]}{\cdots \set{\a[r]}{\text{$V(\a[1], \dotsc, \a[r])$ has rank $r$}} 
	\in p \cdots} \in p. 
\]
Such an $r$ has to exist, because we are working inside a fixed group of finite rank;
if not even $\set{\a[1]}{\text{$V(\a[1])$ has rank one}}$ is in $p$, we set $r=0$ to
keep the notation consistent. Whenever $\a[1], \dotsc, \a[r]$ are taken, in the
correct order, from these nested sets, the rank of the group $V(\a[1], \dotsc,
\a[r])$ is $r$. 

With $r$ being defined in that manner, one also has
\[
	\set{\a[1]}{\cdots \set{\a[r+1]}{\text{$V(\a[1], \dotsc, \a[r+1])$ has rank less 
	than $r$}} \in p \cdots} \in p.
\]
Intersecting with the previous set and using that $p$ is an ultrafilter, we obtain
\begin{align*}
	\set{\a[1]}{\cdots \set{\a[r]}{&\text{$V(\a[1], \dotsc, \a[r])$ has rank $r$
	and} \\
	&\set{\b}{\text{$V(\a[1], \dotsc, \a[r], \b)$ also has rank $r$}} \in p} \in p
	\cdots} \in p.
\end{align*}
But if $V(\a[1], \dotsc, \a[r])$ and $V(\a[1], \dotsc, \a[r], \b)$ both have rank $r$, 
it means that some nonzero multiple of $\b$ has to lie in the first group. We can 
therefore conclude from the previous line that 
\begin{equation} \label{eq:BigSet} \begin{split}
	\set{\a[1]}{\cdots \set{\a[r]}{&\text{$V = V(\a[1], \dotsc, \a[r])$ has rank $r$ 
	and} \\
	&\set{\b}{\text{$N \cdot v(\b) \in V$ for some $N \neq 0$}} \in p} \in p 
	\cdots} \in p.
\end{split} \end{equation}

Finally, let $\mathcal{A}$ denote the set of $r$-tuples $\bigl( \a[1], \dotsc, \a[r] \bigr)$, 
taken in the right order from the nested sets in \eqref{eq:BigSet}; for any one of them, 
the group $V = V(\a[1], \dotsc, \a[r])$ has rank $r$ and the set $\set{\b}{\text{$N 
\cdot v(\b) \in V$ for some $N \neq 0$}}$ is $p$-big.

We have now arrived at our destination---we shall use the splitting $\Hil = \Hil_1 \oplus 
\Hil_1^{\bot}$, where
\[
	\Hil_1 = \bigcap_{\mathcal{A}} \ker P_{\a[1]}\dotsm P_{\a[r]} 
		\qquad \text{and} \qquad
	\Hil_1^{\bot} = \overline{\sum_{\mathcal{A}} \im P_{\a[1]}\dotsm P_{\a[r]}}.
\]
This ends the third step; the proof, based on this splitting, that $P$ is an orthogonal 
projection is contained in the remaining part of the proof.

\subsection*{\texorpdfstring{IV}{Step IV}} 

It remains to prove that the operator $P$ really is a projection. Because of
the splitting from \ref{item:III}, we have two subspaces to consider. Let us begin
with the one that is easier to handle, and show that $P$ is zero on $\Hil_1$. 
If $x \in \Hil_1$, we have $P_{\a[1]} \dotsm P_{\a[r]} x = 0$ for $p$-many 
$\a[1], \dotsc, \a[r]$, thus
\[
	P^r x = \plim{p}_{\a[1]} \cdots \plim{p}_{\a[r]} \Bigl( \unop{g_i(\a[1]) + 
	\dotsb + g_i(\a[r])} \Bigr) P_{\a[1]} \dotsm P_{\a[r]} x = 0.
\]
As $P$ is self-adjoint, $P^r x = 0$ quickly leads to $P x = 0$. 

The complementary subspace $\Hil_1^{\bot}$, on the other hand, requires more
attention. Here, we shall show that $P$ is equal to another projection $P'$,
to be defined below, and constructed with the help of the inductive assumptions.
So suppose that $x \in \im P_{\a[1]} \dotsm P_{\a[r]}$ for a certain tuple $\bigl(
\a[1], \dotsc, \a[r] \bigr) \in \mathcal{A}$; we shall reason that $P x = P' x$,
which, by the usual density argument, is sufficient for equality on all of $\Hil_1^{\bot}$. 

Since $x$ lies in the image of the product $P_{\a[1]} \dotsm P_{\a[r]}$, we get
$P_{\a[k]} x = x$ for each $k = 1, \dotsc, r$. Apply
Lemma~\ref{lem:StrongConvergence} to get strong convergence, in the form
\[
	\plim{p}_{\z[2]} \cdots \plim{p}_{\z[s]} \Norm{\unop{f_i'(\a[k], \z[2], 
	\dotsc, \z[s])} x - x} = 0.
\]
One easily derives that for any $r$ integers $N_1, \dotsc, N_r$,
\[
	\plim{p}_{\z[2]} \cdots \plim{p}_{\z[s]} \Norm{\unop{N_1 f_i'(\a[1], \z[2], 
	\dotsc, \z[s]) + \dotsm + N_r f_i'(\a[r], \z[2], \dotsc, \z[s])} x - x} = 0.
\]
The vectors $v(\a[k])$ span the group $V = V(\a[1], \dotsc, \a[r])$, and so we can
conclude that the equality
\begin{equation} \label{eq:StrongLimit}
	\plim{p}_{\z[2]} \cdots \plim{p}_{\z[s]} \Norm{\unop{h_i(\z[2], \dotsc, \z[s])} 
	x - x} = 0
\end{equation}
is true for any element $(h_1, \dotsc, h_m) \in V$.

We already saw, after \eqref{eq:BigSet}, that the set
\[
	\set{\b}{\text{$N \cdot v(\b) \in V$ for some $N \neq 0$}};
\]
is $p$-big, because of how the integer $r$ was chosen.
As a consequence of Lemma~\ref{lem:Key}, which applies because $\dim p = n$, the set
\[
	\set{\c}{v(\c) - v(0) \in V}
\]
is now also $p$-big. Together with \eqref{eq:StrongLimit}, this gives
\[
	\plim{p}_{\z[2]} \cdots \plim{p}_{\z[s]} \Norm{\unop{f_i'(\c, \z[2], \dotsc, 
	\z[s]) - f_i'(0, \z[2], \dotsc, \z[s])} x - x} = 0
\]
for $p$-many $\c \in \Zn$, hence
\[
	\plim{p}_{\c} \plim{p}_{\z[2]} \cdots \plim{p}_{\z[s]} \Norm{\unop{f_i'(\c, 
	\z[2], \dotsc, \z[s])} x - \unop{f_i(0, \z[2], \dotsc, \z[s])} x} = 0.
\]

The following computation now ends the proof of the fourth and last step:
\begin{align*}
	P x &= \plim{p}_{\c} \Bigl( \unop{g_i(\c)} \Bigr) P_{\c} x = \\
		 &= \plim{p}_{\c} \Bigl( \unop{g_i(\c)} \Bigr) \plim{p}_{\z[2]} \cdots 
		 	\plim{p}_{\z[s]} \unop{f_i'(\c, \z[2], \dotsc, \z[s])} x \\
		 &= \Bigl( \plim{p}_{\c} \unop{g_i(\c)} \Bigr) \Bigl( \plim{p}_{\z[2]} 
		 \cdots \plim{p}_{\z[s]} \unop{f_i(0, \z[2], \dotsc, \z[s])} \Bigr) x
\end{align*}
By the hypotheses $(1,d)$ and $(s-1,d)$, both bracketed expressions define orthogonal
projections, which moreover commute with each other. Let $P'$ be the projection
operator defined as their product; then
we have shown that $P x = P' x$, for all $x \in \im P_{\a[1]} \dotsm P_{\a[r]}$. This
proves that $P$ is also an orthogonal projection when restricted to $\Hil_1^{\bot}$,
and thus completes the proof.
\end{proof}

\subsection*{Comparison with the original}

Since the main purpose of this paper is to reprove the result of \cite{BFM} using
ultrafilters, it may be worthwhile to compare the proof there with the one just
given. The overall argument is the same---a proof by induction, relying on a
splitting defined in terms of certain groups, which depend on the data (the
given IP-sets in \cite{BFM}, the idempotent ultrafilter here). To make the inductive
step work out, we have to allow for multiple $p$-limits, necessitating a more
intricate argument; our proof is consequently slightly longer than the original one.

We should point out that the third step of the proof uses the same splitting as that
in \cite{BFM}. The notion of dimension, mentioned before, is also apparent in the
original paper, and has to be dealt with in much the same way. While we use
Lemma~\ref{lem:Dimension} for this purpose, Bergelson, Furstenberg, and McCutcheon
rely on the Milliken-Taylor theorem to handle the different possible dimensions in a
unified manner. Several other auxiliary results, proved or quoted in the other paper,
also occur at some point in our proof.

Lastly, IP-limits have been replaced by limits along ultrafilters, which means that no
subsequences (or more strictly sub-IP-rings) have to be chosen to get convergence. This
adds much convenience to the argument.

\subsection*{An IP-version} 

As in the example in Section~\ref{sec:II}, we can derive from the previous theorem a
version with IP-sets; because of the many subscripts and superscripts, it is more
complicated to write down.

\begin{thmsec} \label{thm:MainB}
For $j=1, \dotsc, s$, let $q_j \in \beta\mathcal{F}$ be an uncongested idempotent,
and let $U_1, \dotsc, U_m$ be commuting unitary operators on a Hilbert space $\Hil$.
Given $m$ polynomials $f_1, \dotsc, f_m \in \polint{\z[1], \dots,\z[s]}$---with
$\z[j]$ of dimension $n_j$---satisfying $f_i(0) = 0$ for all $i = 1, \dotsc, m$,
and given additionally IP-sets $w_{\bullet}^{k,j}$, indexed by $j = 1, \dotsc, s$ and
$k=1, \dotsc, n_j$, define an operator $P$ on $\Hil$ by
\[
	P = \plim{q_1}_{\alph[1]} \cdots \plim{q_s}_{\alph[s]} 
	\unop{f_i(w_{\alph[1]}^{1, 1}, \dotsc, w_{\alph[1]}^{n_1,1}, \dotsc, 
						w_{\alph[s]}^{1,s}, \dotsc, n_{\alph[s]}^{n_s,s})}  
\]
Then $P$ is always an orthogonal projection. Any two of the operators defined in
this way commute.
\end{thmsec}

\begin{proof} 
We use Lemma~\ref{lem:phistar} and define the following maps. For each $j=1, \dotsc,
s$, let 
\[ 
	\phi_j \colon \mathcal{F} \to \ZZ^{n_j}, \qquad  
	\phi_j(\alpha) = \bigl( w_{\alpha}^{1,j}, \dotsc, w_{\alpha}^{n_j,j} \bigr), 
\] 
and introduce new ultrafilters 
\[
	p_j = \phi_{j\ast}(q_j) \in \beta\Z^{n_j}.
\]
Since the original ultrafilters were uncongested, all $p_j$ are idempotents by virtue
of Lemma~\ref{lem:uncongestedidp}, and we obtain 
\[
	P = \plim{p_1}_{\z[1]} \cdots \plim{p_s}_{\z[s]} \unop{f_i(\z[1], \dotsc, \z[s])}
\]
from Lemma~\ref{lem:phistar}. The result now follows from the previous theorem. 
\end{proof}

Now Lemma~\ref{lem:Uncongested} states that any IP-ring is contained in an
uncongested idempotent of $\beta\mathcal{F}$; it follows that the conclusion of
Theorem~\ref{thm:MainB} holds equally well after replacing each ultrafilter limit by
a limit over some IP-ring. We thus recover the main theorem of the original paper
\cite{BFM}, as we had set out to do.

\section{Consequences} \label{sec:VI}

From the two theorems in the previous section, we can now derive several other
results. In order to simplify the statements, we shall only consider single
$p$-limits.  Let us begin by showing why it is useful that the weak operator limits
we considered are orthogonal projections.

\begin{thmsec} \label{thm:Consequence}
Let $\msp$ be a probability measure space, and let $A \subseteq X$ be a measurable
set. Let $T_1, \dotsc, T_m$ be commuting invertible
measure-preserving transformations on $X$. Furthermore, assume that polynomials
$f_1, \dotsc, f_m \in \polint{\z}$ are given, where $\z = (z_1, \dotsc, z_n)$,
such that $f_i(0) = 0$ for all $i = 1, \dotsc, m$.
\begin{enumerate}
\item For any idempotent $p \in \beta\Zn[n]$, one has
\[
	\plim{p}_{\z} \Ms{A \cap \mspt{f_i(\z)} A} \geq \ms{A}^2.
\]
\item For any uncongested idempotent $q \in \beta\mathcal{F}$ and for any
collection of IP-sets $w_{\bullet}^j$, with $j=1, \dotsc, n$, one has
\[
	\plim{q}_{\alpha} \Ms{A \cap \mspt{f_i(w_{\alpha}^1, \dotsc,
		w_{\alpha}^n)} A} \geq \ms{A}^2.
\]
\end{enumerate}
\end{thmsec}

\begin{proof}
We shall only prove the first statement; the argument will likely be familiar to the
reader anyway. On the Hilbert space $\Hil = \Lzwo$, introduce $m$ commuting unitary 
operators $U_1, \dotsc, U_m$, defining $U_i$ by the rule $U_i g = g \circ T_i$ for $g
\in \Hil$.  By virtue of Theorem~\ref{thm:MainA}, the operator
\[
	P = \plim{p}_{\z} \unop{f_i(\z)}.
\]
is an orthogonal projection onto some closed subspace of $\Hil$.
Write $g$ for the characteristic function of the set $A$, and introduce the
abbreviation
\[
	T = \prod_{i=1}^m T_i^{f_i(\z)}.
\]
Then we have
\[
	\plim{p}_{\z} \mu\bigl(A \cap T^{-1} A\bigr) =
	\plim{p}_{\z} \Inner{g}{\unop{f_i(\z)}g} = \inner{g}{Pg}.
\]
Now $P$ is a projection; therefore, if $e \equiv 1$ denotes the function
identically equal to 1, of norm $\norm{e} = 1$,
\[
	\inner{g}{Pg} = \norm{Pg}^2 = \norm{Pg}^2 \norm{e}^2 \geq \inner{Pg}{e}^2 =
		\inner{g}{Pe}^2.
\]
Finally, $Pe = e$, since $e$ is invariant under the action of the unitary 
operators $U_i$, and so
\[
	\inner{g}{Pe}^2 = \inner{g}{e}^2 = \ms{A}^2.
\]
Combining the three displayed (in)equalities gives the desired result.
\end{proof}

The consequences of the preceding theorem are twofold. First, when applied to the case of
a single measure-preserving transformation, the two inequalities in
Theorem~\ref{thm:Consequence} show precisely that the sets
\[
	\set{f(\z)}{\z \in \Zn[n]}
\]
and
\[
	\set{f(w_{\alpha}^1, \dotsc, w_\alpha^n)}{\alpha \in \mathcal{F}}
\]
are sets of nice recurrence; here $f$ may be any polynomial in $\polint{\z}$
satisfying $f(0) = 0$, and $w_{\bullet}^j$, with $j = 1, \dotsc, n$, can
be arbitrary IP-sets.

Secondly, we can exploit the fact that the idempotent ultrafilters in 
Theorem~\ref{thm:Consequence} may be chosen arbitrarily. Under the assumptions
made above, for any $\epsilon > 0$, the set
\[
	R_{\epsilon} = \set{\z \in \Zn[n]}{\mu\bigl(A \cap T^{-1} A\bigr) \geq \ms{A}^2 -
		\epsilon}
\]
has to be contained in every idempotent ultrafilter in $\beta\Zn[n]$. The reader will
remember that this is equivalent to saying that $R_{\epsilon}$ is IP*, that is,
intersects every IP-set of $\Zn[n]$. In particular, every $R_{\epsilon}$ is a
syndetic set, because the IP* property implies syndeticity. A special case of this
result is Khintchine's recurrence theorem, which states that for a single
measure-preserving transformation $T$, the sets of nice returns
\[
	\set{n \in \N}{\mu\bigl(A \cap T^{-1} A\bigr) \geq \ms{A}^2 - \epsilon}
\]
are syndetic. Several other applications may be found in the original paper
\cite{BFM}.

\end{document}